\documentclass[final]{dmtcs-episciences}
\usepackage[round]{natbib}

\usepackage[active]{srcltx}
\usepackage{latexsym,amssymb}
\usepackage{amsmath}
\usepackage{amsthm}
\usepackage{stmaryrd}
\usepackage{enumerate}
\usepackage[latin1]{inputenc}
\usepackage[all,pdf]{xy}

\renewcommand{\ge}{\geqslant}
\renewcommand{\le}{\leqslant}
\let\op=\llbracket
\let\cl=\rrbracket
\def\pv#1{\ensuremath{\mathsf{#1}}}
\def\Om#1#2{\ensuremath{\overline{\Omega}_{#1}{\pv{#2}}}}

\newcommand\malcev{%
\mathbin{\hbox{$\bigcirc$\kern-9pt\raise1.2pt\hbox{\scriptsize$m$}\,}}}
\newcommand\smalcev{%
\mathbin{\hbox{$\bigcirc$\kern-8pt\raise1.2pt\hbox{\tiny$m$}\,}}}
\def\Pol#1{\ensuremath{\mathop{\mathrm{Pol}}\pv{#1}}}

\let\cal=\mathcal
\def\Cl#1{\ensuremath{\cal#1}}

\newtheorem{Thm}{Theorem}[section]
\newtheorem{Prop}[Thm]{Proposition}
\newtheorem{Lemma}[Thm]{Lemma}
\newtheorem{Cor}[Thm]{Corollary}

\newtheorem{Problem}[Thm]{Problem}
\newtheorem{Conjecture}[Thm]{Conjecture}
\theoremstyle{definition}

\newtheorem{Remark}[Thm]{Remark}

\title{On the insertion of $n$-powers}

\author{Jorge Almeida\affiliationmark{1}\thanks{Partial support by
    CMUP (UID/MAT/ 00144/2019) which is funded by FCT (Portugal) with
    national (MCTES) and European structural funds (FEDER) under the
    partnership agreement PT2020.} %
  \and Ond{\v r}ej Kl{\'i}ma\affiliationmark{2}\thanks{Supported by
    Grant 15-02862S of the Grant Agency of the Czech Republic.} }%
\affiliation{CMUP, Dep.\ Matem\'atica, Faculdade de Ci\^encias,
  Universidade do Porto, Portugal\\
  Dept.\ of Mathematics and Statistics, Masaryk University, Czech
  Republic}


\received{2017-11-16}
\revised{2019-1-16}
\accepted{2019-1-16}

\begin{document}

\publicationdetails{21}{2019}{3}{5}{4072}
\maketitle
\begin{abstract}
  In algebraic terms, the insertion of $n$-powers in words may be
  modelled at the language level by considering the pseudovariety of
  ordered monoids defined by the inequality $1\le x^n$. We
  compare this pseudovariety with several other natural
  pseudovarieties of ordered monoids and of monoids associated with
  the Burnside pseudovariety of groups defined by the identity
  $x^n=1$. In particular, we are interested in determining the
  pseudovariety of monoids that it generates, which can be viewed as
  the problem of determining the Boolean closure of the class of
  regular languages closed under $n$-power insertions. We exhibit a
  simple upper bound and show that it satisfies all pseudoidentities
  which are provable from $1\le x^n$ in which both sides are regular
  elements with respect to the upper bound.
\end{abstract}

\makeatletter%
\@namedef{subjclassname@2010}{%
  \textup{2010} Mathematics Subject Classification}%
\makeatother



\maketitle

\section{Introduction}
\label{sec:intro}

Pseudovarieties of ordered monoids have been introduced in the theory
of finite semigroups as a tool that, via a generalization of
Eilenberg's correspondence and the syntactic monoid, provides a
classifier for classes of regular languages,
cf.~\cite{Pin:1995a,Pin:1997}. More generally than varieties of
languages, they classify the so-called positive varieties of
languages. As for the original version of Eilenberg's correspondence
for pseudovarieties of monoids, the extended version prompted
additional interest in studying pseudovarieties of ordered monoids,
particularly in the context of concatenation hierarchies of regular
languages, which provided the initial motivation for introducing them.

Even before pseudovarieties of ordered monoids were considered,
ordered monoids had already been shown to play a role in the theory of
finite semigroups. A notable instance is a direct algebraic proof of
the fact that every finite \Cl J-trivial monoid is a quotient of some
finite ordered monoid satisfying the inequality $1\le x$, see
\cite{Straubing&Therien:1988a}, a fact that turns out to be equivalent
to Simon's characterization of piecewise testable languages as those
whose syntactic monoid is a finite \Cl J-trivial monoid, see
\cite{Simon:1975}, one of the classical results that led to the
formulation of Eilenberg's correspondence, cf.~\cite{Eilenberg:1976}.
Note that the language counterpart of the pseudovariety of ordered
monoids defined by the inequality $1\le x$ is the class of regular
languages that are closed under inserting arbitrary words in each of
their elements.

Another important instance of an inequality of the form $1\le
x^\alpha$ is the weakest such inequality, namely $1\le x^\omega$. The
pseudovariety of monoids generated by the class of ordered monoids it
defines was the object of deep research in the 1980's which led to
many alternative descriptions, from block groups to power groups, as
well as the language counterpart given by the Boolean-polynomial
closure of the class of all group languages. A discussion of such
results, whose key ingredient is due to \cite{Ash:1991}, may be
found in~\cite{Henckell&Margolis&Pin&Rhodes:1991,Pin:1995}. Most
descriptions of that pseudovariety involve some construction on the
pseudovariety \pv G of all finite groups, such as power groups
(\pv{PG}), the semidirect product $\pv J*\pv G$, the Mal'cev product
$\pv J\malcev\pv G$, and block groups (\pv{BG}). The relationships
between such constructions starting from an arbitrary pseudovariety of
groups instead of the pseudovariety~\pv G have
been extensively studied by Auinger and Steinberg, see
\cite{Steinberg:1999f, %
  Steinberg:1999c, %
  Steinberg:1999a, %
  Auinger&Steinberg:2001b, %
  Auinger&Steinberg:2001a, %
  Auinger&Steinberg:2005c, %
  Auinger&Steinberg:2003a}. %
In particular, the situation is radically different from the
well-behaved case of~\pv G for the Burnside pseudovariety defined by
the identity $x^n=1$ for $n\ge2$.

The aim of our paper is to investigate the pseudovariety of ordered
monoids $\op 1\le x^n\cl$ defined by the inequality $1\le x^n$, which
is the algebraic counterpart of the positive variety of languages
closed under the insertion of $n$-powers. We are also interested in
the Boolean closure of that positive variety, for which decidability
of membership remains an open problem. It corresponds to the
pseudovariety of monoids generated by $\op 1\le x^n\cl$, which may be
viewed as an extension of the case $1\le x$ and a restriction of the
case $1\le x^\omega$ by bounding the exponent. We compare these
pseudovarieties with the classical constructions on the corresponding
Burnside pseudovariety, defined by $x^n=1$, and with the best upper
bound we have been able to find. This is the pseudovariety
$(\pv{BG})_n$ of block groups defined by the pseudoidentity
\begin{math}
  (xy^\omega z)^{\omega+1}=(xy^nz)^{\omega+1}.
\end{math}
We also propose an ordered version of the pseudoidentity proof scheme
introduced by the authors, see~\cite{Almeida&Klima:2017a}. Finally, we
show that if a pseudoidentity over~$(\pv{BG})_n$ whose sides are
regular pseudowords may be proved from $1\le x^n$, then it is trivial,
which gives some evidence towards our upper bound being optimal.

\section{Background}
\label{sec:background}

We assume the reader is familiar with the basics of finite semigroup
theory, particularly with pseudovarieties, pseudoidentities, and
relatively free profinite monoids. For details, see
\cite{Pin:1986;bk, %
  Almeida:1994a, %
  Almeida:2003cshort, %
  Rhodes&Steinberg:2009qt, %
  Almeida&Costa:2015hb}. %
In particular, recall that a \emph{profinite monoid} is a compact
zero-dimensional monoid. For a pseudovariety \pv V of monoids, the
pro-\pv V monoid freely generated by a set $A$ is denoted \Om AV. By a
\emph{\pv V-pseudoidentity} we mean a formal equality $u=v$ with
$u,v\in\Om AV$ for some finite set $A$. For a set $\Sigma$ of \pv
V-pseudoidentities, the class of all monoids from~\pv V satisfying all
pseudoidentities from~$\Sigma$ is denoted $\op\Sigma\cl$. Most often,
we consider \pv M-pseudoidentities, where \pv M is the pseudovariety
of all finite monoids. Elements of~\Om AM are sometimes called
\emph{pseudowords}.

For an element $s$ of~a profinite monoid,
$s^\omega$ denotes the unique idempotent in the closed subsemigroup
$\overline{\langle s\rangle}$ generated by~$s$, while $s^{\omega-1}$
denotes the inverse of $s^{\omega+1}=s^\omega s$ in the unique maximal
subgroup of~$\overline{\langle s\rangle}$. For a nonzero integer~$k$,
$s^{\omega+k}$ stands for $(s^{|k|})^{\omega+\varepsilon}$, where
$\varepsilon$ is the sign of~$k$.

By an \emph{ordered monoid} we mean a monoid with a partial order that
is compatible with the monoid operation, cf.~\cite{Pin:1997}. A
pseudovariety of ordered monoids is a nonempty class of such
structures that is closed under taking images under order-preserving
homomorphisms, subsemigroups under the induced order, and finite
direct products. The theory of pseudovarieties of ordered monoids is a
natural extension of the unordered case. For a pseudovariety \pv V of
ordered monoids, forgetting the order of its elements, we may consider
the pseudovariety of monoids $\langle\pv V\rangle$ it generates. On
the other hand, an unordered monoid may be viewed as an ordered one
under the trivial partial order, given by equality in the monoid. For
a pseudovariety of monoids \pv V, the pseudovariety of ordered monoids
$\pv V'$ that the members of~\pv V generate when ordered trivially
consists precisely of all monoids in~\pv V under all possible
compatible orders. Thus, it is natural to identify \pv V and $\pv V'$
and we do so it freely throughout this paper.

By a \emph{(pseudo)inequality} we mean a formal inequality $u\le v$
with $u,v\in\Om AM$ for some finite set~$A$. The class of all finite
ordered monoids satisfying a given set $\Sigma$ of inequalities is
also denoted $\op\Sigma\cl$.

For a pseudovariety \pv V of ordered monoids, there is also a pro-\pv
V monoid freely generated by a set $A$, denoted \Om AV which, as a
topological monoid, coincides with $\Om A{}\langle\pv V\rangle$. It
may be viewed as the quotient of~\Om AM by the (compatible closed)
quasiorder $\le$ defined by $u\le v$ when \pv V satisfies the
inequality $u\le v$.

\textit{Mutatis mutandis}, instead of monoids one may consider
semigroups. For a pseudovariety \pv V of monoids, we usually also
denote by \pv V the pseudovariety of semigroups it generates.
Occasionally, we refer to pseudovarieties of ordered semigroups.

There are several pseudovarieties that play an important role in this
paper. Among them are the pseudovariety \pv J of all finite \Cl
J-trivial monoids, the pseudovariety \pv A of all finite aperiodic
monoids, and the pseudovariety \pv{Sl} of all finite semilattices.
Some operators on pseudovarieties are also relevant. For a
pseudovariety \pv V of semigroups, \pv{EV} denotes the pseudovariety
of all finite monoids whose idempotents generate a semigroup from~\pv
V, \pv{DV} denotes the pseudovariety of all finite semigroups whose
regular \Cl D-classes are subsemigroups from~\pv V, and \pv{PV}
denotes the pseudovariety generated by all power semigroups of the
semigroups from~\pv V. For a pseudovariety of groups \pv H, \pv{BH}
and $\overline{\pv H}$ denote the pseudovarieties of all finite
monoids whose blocks are groups from~\pv H and whose subgroups belong
to~\pv H, respectively. When \pv V is a pseudovariety of ordered
semigroups and \pv W is a pseudovariety of monoids, the Mal'cev
product $\pv V\malcev\pv W$ consists of all finite ordered monoids for
which there is a relational morphism into a monoid from~\pv W such
that the preimage of each idempotent is a member of~\pv V.

Given a language $L$ over a finite alphabet $A$, meaning a subset of
the free monoid $A^*$, the associated \emph{syntactic quasiorder} is
the quasiorder on $A^*$ defined by $u\le v$ if, for all $x,y\in A^*$,
$xuy\in L$ implies $xvy\in L$.\footnote{In the literature, one often
  finds the syntactic quasiorder defined to be the reverse quasiorder
  (see \cite{Almeida&Cano&Klima&Pin:2015} for historical details).}
The \emph{syntactic ordered monoid} of $L$, denoted
$\mathrm{Synt}(L)$, is the quotient ordered monoid by the quasiorder
$\le$, meaning the quotient of $A^*$ by the congruence
${\le}\cap{\ge}$, endowed with the partial order induced by~$\le$. By
the \emph{syntactic monoid} of~$L$ we mean the same monoid
$\mathrm{Synt}(L)$ but with no reference to the order.

\section{Preliminary results}
\label{sec:prelim-results}

Consider the pseudovarieties $\pv J^+=\op 1\le x\cl$ and
$\pv{LI}^+=\op x^\omega\le x^\omega yx^\omega\cl$, respectively of
ordered monoids and of ordered semigroups. By
\cite[Theorem~5.9]{Pin&Weil:1994c}, for a pseudovariety of monoids \pv
V, the polynomial closure\footnote{meaning the pseudovariety of
  ordered monoids \Pol V corresponding to the positive variety of
  languages generated by the class of languages which, for a finite
  alphabet $A$, consists of the products of the form $L_0a_1L_1\cdots
  a_nL_n$, where the $a_i\in A$ and the $L_i$ are \pv V-languages.}
of~\pv V is the pseudovariety of ordered monoids
\begin{math}
  \Pol V=\pv{LI}^+\malcev\pv V.
\end{math}
As was proved by~\cite{Pin&Weil:1996a}, $\pv{LI}^+\malcev\pv V$ is
defined by the inequalities of the form
\begin{math}
  u^\omega\le u^\omega vu^\omega
\end{math}
such that the pseudoidentities
\begin{math}
  u=v=v^2
\end{math}
hold in~\pv V. In particular, in case \pv V~is a pseudovariety of
groups, one may take $u=1$, so that the defining inequalities for \Pol
V are reduced to $ 1\le v$ whenever \pv V satisfies $v=1$. This
observation proves the following statement.

\begin{Lemma}[{\cite[Corollary~3.1]{Steinberg:1999f}}]
  \label{l:PolH}
  If \pv H is a pseudovariety of groups, then
  \begin{math}
    \Pol H=\pv J^+\malcev\pv H.
  \end{math}
  \qed
\end{Lemma}

The following result allows us to separate two pseudovarieties of
interest.

\begin{Lemma}
  \label{l:separate-J+mHn-from-1LEQxn}
  For $n\ge2$, $\op 1\le x^n\cl$ is not contained in
  \begin{math}
    \pv J^+\malcev\op x^n=1\cl.
  \end{math}
\end{Lemma}

\begin{proof}
  Consider first the case where $n\ge3$. In the Burnside
  pseudovariety~$\op x^n=1\cl$, we have
  \begin{math}
    (y^{n-1}x)^{n-1}(xy)^{n-1}x^2=x^{n-1}yy^{n-1}x^{n-1}x^2=1.
  \end{math}
  Hence,
  the inequality
  \begin{math}
    1\le (y^{n-1}x)^{n-1}(xy)^{n-1}x^2
  \end{math}
  holds in the Mal'cev product
  \begin{math}
    \pv J^+\malcev\op x^n=1\cl.
  \end{math}
  Let $L$ be the language over the alphabet $A=\{x,y,t\}$
  given by
  \begin{displaymath}
    L=\{pu^nq\in A^*: u\in A^+,\,p,q\in A^*\}
    \cup\{w\in A^*:|w|\ge(n+1)^2\}.
  \end{displaymath}
  Then, $L$ is a cofinite language, whence it is regular. Since $u^n$
  appears in~$L$ in every context, the syntactic ordered monoid
  $\mathrm{Synt}(L)$ satisfies the inequality $1\le x^n$. Note also
  that
  \begin{math}
    t\cdot1\cdot t^{n-1}
  \end{math}
  belongs to~$L$ but, since $n\ge3$,
  \begin{math}
    t\cdot(y^{n-1}x)^{n-1}(xy)^{n-1}x^2\cdot t^{n-1}
  \end{math}
  does not as it is a word of length
  \begin{math}
    (n-1)(n+3)+3=n^2+2n
  \end{math}
  which does not contain any factor of the form $u^n$ with $u\ne1$.
  Hence, $\mathrm{Synt}(L)$ fails the inequality
  \begin{math}
    1\le (y^{n-1}x)^{n-1}(xy)^{n-1}x^2.
  \end{math}

  In case $n=2$, we consider instead the inequality $1\le xyzxzy$. Let
  $A=\{x,y,z,t\}$ and consider the language
  \begin{displaymath}
    L=\{pu^2q\in A^*: u\in A^+,\,p,q\in A^*\}
    \cup\{w\in A^*:|w|\ge 9\}.
  \end{displaymath}
  The argument proceeds as in the previous case, where the essential
  ingredient that needs to be noted is that the word $txyzxzyt$ has no
  square factor.
\end{proof}

\begin{Cor}
  \label{c:1LEQxn-non-PolV}
  For $n\ge2$, the pseudovariety $\op1\le x^n\cl$ is not of the form
  \Pol V for any pseudovariety of monoids~\pv V.
\end{Cor}

\begin{proof}
  Let \pv V be a pseudovariety of monoids and suppose that
  \begin{math}
    \op1\le x^n\cl=\Pol V.
  \end{math}
  Since
  \begin{math}
    \pv V\subseteq\Pol V,
  \end{math}
  it follows that \pv V satisfies the inequality $1\le x^n$, whence
  also the identity $x^n=1$ so that, in particular, \pv V must be a
  pseudovariety of groups. By Lemma~\ref{l:PolH}, we deduce that
  \begin{math}
    \Pol V=\pv J^+\malcev\pv V.
  \end{math}
  By Lemma~\ref{l:separate-J+mHn-from-1LEQxn},
  $\pv J^+\malcev\pv V$ satisfies a pseudoidentity that fails in
  $\op1\le x^n\cl$, which entails
  \begin{math}
    \op1\le x^n\cl\nsubseteq\Pol V,
  \end{math}
  in contradiction with the initial assumption.
\end{proof}

In contrast with Corollary~\ref{c:1LEQxn-non-PolV}, for the
pseudovariety \pv G of all finite groups, \cite[Theorem~5.9
and~2.7]{Pin&Weil:1994c} yield the equalities
\begin{math}
  \Pol G=\pv{LI}^+\malcev\pv G=\op 1\le x^\omega\cl.
\end{math}
It is important to notice that the known proof of the latter equality
depends on a deep theorem of~\cite{Ash:1991}. On the other hand, for
$n=1$, for the trivial pseudovariety
\begin{math}
  \pv I=\op x=1\cl,
\end{math}
we have
\begin{math}
  \Pol I=\pv J^+\malcev\pv I=\pv J^+=\op 1\le x\cl.
\end{math}

Next, we recall some related results.

\begin{Thm}[\cite{Higgins&Margolis:1999}]
  \label{t:Higgins-Margolis}
  The pseudovariety \pv G is the only pseudovariety of groups \pv H
  such that $\pv A\cap\pv{ESl}$ is contained in $\pv{DA}\malcev\pv H$.
\end{Thm}

The following is an immediate application of the preceding theorem
which has already been observed in~\cite[Proposition~13]{Steinberg:1999a}.

\begin{Cor}
  \label{c:JmH-vs-BH}
  For a pseudovariety of groups \pv H, the equality
  \begin{math}
    \pv J\malcev\pv H=\pv{BH}
  \end{math}
  holds if and only if $\pv H=\pv G$.
\end{Cor}

\begin{proof}
  The equalities
  \begin{math}
    \pv J\malcev\pv G=\pv{BG}=\pv{EJ}
  \end{math}
  are well-known
  \cite[Theorem~4.7]{Margolis&Pin:1985}, see also
  \cite[Proposition~7.4]{Pin:1995}. For the converse, it suffices to
  observe that
  \begin{math}
    \pv{BH}=\pv{EJ}\cap\overline{\pv H}\supseteq\pv{ESl}\cap\pv A
  \end{math}
  while
  \begin{math}
    \pv J\malcev\pv H\subseteq\pv{DA}\malcev\pv H
  \end{math}
  and apply
  Theorem~\ref{t:Higgins-Margolis} to conclude that, if %
  \begin{math}
    \pv{BH}\subseteq\pv J\malcev\pv H,
  \end{math}
  then $\pv H=\pv G$.
\end{proof}

The following theorem summarizes several results
of~\cite{Steinberg:1999f} related with our present purposes. In it,
a pseudovariety of groups \pv H is said to be \emph{arborescent} if it
satisfies the equality
\begin{math}
  (\pv H\cap\pv{Ab})*\pv H=\pv H,
\end{math}
where \pv{Ab} is the pseudovariety of all finite Abelian
groups; the adjective arborescent comes from tree-like homological
properties of the Cayley graphs of the free pro-\pv H groups, see
\cite{Almeida&Weil:1994c} and also
\cite[Theorem~2.5.3]{Ribes:2017bk}.

\begin{Thm}
  \label{t:Steinberg-IJAC10}
  Let \pv H be a pseudovariety of groups.
  \begin{enumerate}
  \item
    \begin{math}
      \langle\pv J^+*\pv V\rangle=\pv J*\pv V
    \end{math}
    for every pseudovariety of monoids \pv V
    \cite[Proposition~2.2]{Steinberg:1999f}.
  \item
    \begin{math}
      \pv J^+\malcev\pv H\subseteq\pv J*\pv H
    \end{math}
    \cite[Proposition~5.2]{Steinberg:1999f}.
  \item
    \begin{math}
      \pv{PH}\subseteq\pv J*\pv H=\langle{\Pol H}\rangle\subseteq\pv
      J\malcev\pv H
    \end{math}
    \cite[Theorem~5.2 and Corollary~5.2]{Steinberg:1999f}.
  \item In case \pv H is arborescent,
    \begin{math}
      \pv{PH}=\pv J*\pv H=\langle{\Pol H}\rangle=\pv J\malcev\pv H
    \end{math}
    \cite[Theorem~5.9]{Steinberg:1999f}.
  \end{enumerate}
\end{Thm}

A related problem for which only partial answers seem to be known is
the following.

\begin{Problem}
  When does the equality
  \begin{math}
    \langle\pv J^+\malcev\pv V\rangle=\pv J\malcev\pv V
  \end{math}
  hold for a given pseudovariety \pv V of monoids?
\end{Problem}

Note that, from Lemma~\ref{l:PolH} and
Theorem~\ref{t:Steinberg-IJAC10} it follows that, if \pv H is a
pseudovariety of groups, then
\begin{displaymath}
  \Pol H %
  =\pv J^+\malcev\pv H %
  \subseteq \pv J*\pv H=\langle{\Pol H}\rangle %
  \subseteq\pv J\malcev\pv H.
\end{displaymath}
In particular, we get the following result.

\begin{Cor}
  \label{c:BPolH}
  If \pv H is a pseudovariety of groups, then
  $\langle\pv J^+\malcev\pv H\rangle=\pv J*\pv H$.\qed
\end{Cor}

\begin{Remark}
  \label{r:J*H-vs-JmH}
  Note that the inclusion 
  \begin{math}
    \pv J*\pv H \subseteq\pv J\malcev\pv H
  \end{math}
  may be strict. A characterization of the pseudovarieties of groups
  for which equality holds is given
  in~\cite[Theorem~8.3]{Auinger&Steinberg:2001a}. These are the
  so-called arboreous pseudovarieties of groups, defined in terms of
  certain geometrical properties of the Cayley graphs of the
  corresponding relatively free profinite groups. That %
  $\pv J*\pv H\subsetneqq\pv J\malcev\pv H$ %
  for all nontrivial pseudovarieties satisfying some identity of the
  form $x^n=1$ had previously been shown
  in~\cite[Theorem~7.32]{Steinberg:1999c}.
\end{Remark}

In general, we may use the basis theorems for the Mal'cev, see
\cite{Pin&Weil:1996a}, and semidirect products, see
\cite{Almeida&Weil:1996} and also
\cite[Section~3.7]{Rhodes&Steinberg:2009qt}, to obtain bases of
pseudoidentities:
\begin{itemize}
\item using \cite{Knast:1983b} and \cite{Almeida&Weil:1996}, $\pv
  J*\pv H$ is defined by the pseudoidentities of the form %
  \begin{displaymath}
    (xy)^\omega xt(zt)^\omega=(xy)^\omega(zt)^\omega
  \end{displaymath}
  for all pseudowords $x,y,z,t$ such that the pseudoidentities $x=z$
  and $xy=zt=1$ hold in~\pv H;
\item $\pv J\malcev\pv H$ is defined by the pseudoidentities of the
  forms %
  \begin{displaymath}
    u^{\omega+1}=u^\omega \text{ and } (uv)^\omega=(vu)^\omega
  \end{displaymath}
  for all pseudowords $u,v$ such that the pseudoidentities $u^2=u=v$
  hold in~\pv H.
\end{itemize}
In the particular case where
\begin{math}
  \pv H=\op x^n=1\cl,
\end{math}
since this pseudovariety is locally finite by Zelmanov's solution of
the restricted Burnside problem, see \cite{Zelmanov:1991}, we may take
above $x,y,z,t,u,v$ to be words. More precisely the \pv H-free groups
are finite and computable (see the discussion
on~\cite[Page~369]{Rhodes&Steinberg:2009qt}). Hence, given a finite
monoid, to test membership in the pseudovarieties
\begin{math}
  \pv J*\op x^n=1\cl
\end{math}
and
\begin{math}
  \pv J\malcev\op x^n=1\cl,
\end{math}
one only needs to check a computable finite set of effectively
verifiable pseudoidentities in the respective bases above. Indeed,
given a finite monoid $M$, with a generating subset $A$, one only
needs to consider the above pseudoidentities in which the words
$x,y,z,t,u,v$ have length at most
\begin{math}
  |M\times\Om A\op x^n=1\cl|.
\end{math}
Hence, both those product pseudovarieties have decidable membership
problem, which adds to the interest in comparing them with the main
object of our investigation, namely the pseudovariety
\begin{math}
  \langle\op 1\le x^n\cl\rangle.
\end{math}

\section{The pseudovariety \texorpdfstring{$(\pv{BG})_n$}{BGn}}
\label{sec:BGn}

We introduce three alternative bases of pseudoidentities for a
pseudovariety naturally associated with~\pv{BG} and a positive
integer~$n$. Recall that \pv{BG} may be defined by the pseudoidentity
\begin{math}
  (x^\omega y)^\omega=(yx^\omega)^\omega.
\end{math}

\begin{Lemma}
  \label{l:BGn}
  The pseudovarieties %
  \begin{align*}
    \pv U_n&=\op x^{\omega+n}=x^\omega,\
           (xy^n)^\omega=(y^nx)^\omega\cl, \\
    \pv V_n&=\op x^{\omega+n}=x^\omega,\
           (xy^\omega)^\omega=(y^nx)^\omega\cl, \\
    \pv W_n&=\op(xy^\omega z)^{\omega+1}=(xy^nz)^{\omega+1}\cl
           \cap\pv{BG}
  \end{align*}
  coincide.
\end{Lemma}

\begin{proof}
  ($\pv W_n\subseteq\pv U_n$) Taking $x=z=1$ in
  \begin{math}
    (xy^\omega z)^{\omega+1}=(xy^nz)^{\omega+1},
  \end{math}
  we obtain
  \begin{math}
    y^\omega=y^{\omega+n}.
  \end{math}
  If, instead, we take $x=1$, we get
  \begin{math}
    (y^\omega z)^{\omega+1}=(y^nz)^{\omega+1}
  \end{math}
  which, by raising both sides to the power $\omega$, yields
  \begin{math}
    (y^\omega z)^\omega=(y^nz)^\omega.
  \end{math}
  Similarly, we also have
  \begin{math}
    (zy^\omega)^\omega=(zy^n)^\omega
  \end{math}
  in~$\pv W_n$. Since $\pv W_n$ is contained in~\pv{BG}, it also
  satisfies
  \begin{math}
    (y^\omega z)^\omega=(zy^\omega)^\omega.
  \end{math}
  Hence, $\pv W_n$ satisfies
  \begin{math}
    (zy^n)^\omega=(y^nz)^\omega.
  \end{math}

  ($\pv U_n\subseteq\pv V_n$) The proof consists in showing that
  several pseudoidentities hold in~$\pv U_n$. Throughout we write
  equality in the sense of pseudoidentities valid in~$\pv U_n$.
  Substituting $y^\omega$ for~$y$ in the pseudoidentity
  \begin{math}
    (xy^n)^\omega=(y^nx)^\omega,
  \end{math}
  we obtain
  \begin{math}
    (xy^\omega)^\omega=(y^\omega x)^\omega.
  \end{math}
  Associativity gives
  \begin{displaymath}
    (y^nx)^\omega %
    =y^n(xy^n)^\omega(xy^n)^{\omega-1}x %
    =y^n\cdot(y^nx)^\omega\cdot(xy^n)^{\omega-1}x,
  \end{displaymath}
  which, by iteration, yields
  \begin{math}
    (y^nx)^\omega=y^{nk}(y^nx)^\omega\bigl((xy^n)^{\omega-1}x\bigr)^k,
  \end{math}
  so that, taking limits, we get
  \begin{math}
    (y^nx)^\omega=y^\omega(y^nx)^\omega\bigl((xy^n)^{\omega-1}x\bigr)^\omega,
  \end{math}
  whence
  \begin{math}
    (y^nx)^\omega=y^\omega(y^nx)^\omega
  \end{math}
  since $y^\omega$ is an
  idempotent. Similarly, we obtain
  \begin{displaymath}
    (y^nx)^\omega %
    =y^\omega(y^nx)^\omega %
    =y^\omega(xy^n)^\omega %
    =y^\omega x\cdot(y^nx)^\omega\cdot y^n(xy^n)^{\omega-1},
  \end{displaymath}
  whence %
  \begin{math}
    (y^nx)^\omega %
    =(y^\omega x)^k(y^nx)^\omega\bigl(y^n(xy^n)^{\omega-1}\big)^k
  \end{math}
  and
  \begin{math}
    (y^nx)^\omega=(y^\omega x)^\omega(y^nx)^\omega.
  \end{math}
  On the other hand, we have
  \begin{displaymath}
    (y^\omega x)^\omega %
    =y^n(y^\omega x)^\omega %
    =y^n(xy^\omega)^\omega %
    =y^nx\cdot(y^\omega x)^\omega\cdot y^\omega(xy^\omega)^{\omega-1}
  \end{displaymath}
  so that
  \begin{math}
    (y^\omega x)^\omega=(y^nx)^\omega(y^\omega x)^\omega
  \end{math}
  follows as above. Hence, the idempotents $(y^\omega x)^\omega$ and
  $(y^nx)^\omega$ are \Cl R-equivalent. By symmetry, they are also \Cl
  L-equivalent, whence they are equal. In particular, we get
  \begin{math}
    (xy^\omega)^\omega=(y^nx)^\omega.
  \end{math}

  ($\pv V_n\subseteq\pv W_n$) Substituting $y^\omega$ for $y$ in the
  pseudoidentity
  \begin{math}
    (xy^\omega)^\omega=(y^nx)^\omega,
  \end{math}
  we obtain
  \begin{math}
    (xy^\omega)^\omega=(y^\omega x)^\omega,
  \end{math}
  which is a defining pseudoidentity for~\pv{BG}. It remains to show
  that $\pv V_n$ satisfies the pseudoidentity %
  \begin{math}
    (xy^\omega z)^{\omega+1}=(xy^nz)^{\omega+1}.
  \end{math}
  Indeed, it satisfies the following pseudoidentities:
  \begin{align*}
    (xy^n z)^{\omega+1} %
    &=x(y^n zx)^\omega y^n z %
      =x(zxy^\omega)^\omega y^nz \\
    &=x(zxy^\omega)^\omega z %
      =x(zxy^\omega)^\omega y^\omega z \\
    &=x(y^\omega zx)^\omega y^\omega z %
      =(xy^\omega z)^{\omega+1}.
  \end{align*}
  This completes the proof of the lemma.
\end{proof}

From hereon, we denote by $(\pv{BG})_n$ the pseudovariety of
Lemma~\ref{l:BGn}.

We now return to the pseudovariety of ordered monoids $\op 1\le
x^n\cl$.

\begin{Prop}
  \label{p:upperbound-1LEQxn}
  The pseudovariety $\op 1\le x^n\cl$ is contained in~$(\pv{BG})_n$.
\end{Prop}

\begin{proof}
  The pseudovariety $\op 1\le x^n \cl$ is contained in
  \begin{math}
    \op 1\le x^\omega\cl=\Pol G\subseteq\pv{BG}.
  \end{math}
  Moreover, it satisfies the following inequalities:
  \begin{align*}
    xy^\omega z %
    &=x(y^n)^\omega z \\
    &\le x\bigl(y^n(zx)^n\bigr)^{\omega-1}y^nz %
      =x\bigl(y^nz(xz)^{n-1}x\bigr)^{\omega-1}y^nz \\
    &\le x\bigl(y^nz(xy^nz)^{n-1}x\bigr)^{\omega-1}y^nz %
      =(xy^nz)^{n(\omega-1)+1} %
      =(xy^nz)^{\omega-n+1} \\
    &\le (xy^nz)^{\omega+1},
  \end{align*}
  whence
  \begin{math}
    xy^\omega z\le(xy^nz)^{\omega+1}.
  \end{math}
  Raising both sides of the preceding inequality to the power
  $\omega+1$, we obtain
  \begin{math}
    (xy^\omega z)^{\omega+1}\le(xy^nz)^{\omega+1}.
  \end{math}
  For the reverse inequality, just note that $\op 1\le x^n\cl$
  satisfies $y^n\le y^{mn}$ for every positive integer $m$ and,
  therefore, also $y^n\le y^\omega$.
\end{proof}

The following lemma gathers some elementary properties of the
pseudovariety \pv{BG}.

\begin{Lemma}
  \label{l:pseudoidentities-for-BG}
  The pseudovariety \pv{BG} satisfies the following pseudoidentities:
  \begin{itemize}
  \item
    \begin{math}
      (xy^{\omega+1})^\omega %
      =y^{\omega-1}(y^{\omega+1}x)^\omega y^{\omega+1};
    \end{math}
  \item
    \begin{math}
      (xy^\omega z)^\omega(xz)^{\omega+1} %
      =(xy^\omega z)^{\omega+1} %
      =(xz)^{\omega+1}(xy^\omega z)^\omega;
    \end{math}
  \item
    \begin{math}
      (xy^\omega z)^\omega(xz)^\omega %
      =(xy^\omega z)^\omega %
      =(xz)^\omega(xy^\omega z)^\omega;
    \end{math}
  \item
    \begin{math}
      (xy^\omega z)^\omega(xt^\omega z)^\omega = (xy^\omega
      z)^{\omega+1}(xt^\omega z)^{\omega-1}.
    \end{math}
  \end{itemize}
\end{Lemma}

\begin{proof}
  The following pseudoidentities hold in~\pv{BG}:
  \begin{align*}
    (xy^{\omega+1})^\omega %
    &=(xy\,y^\omega)^\omega %
      =(y^\omega\,xy)^\omega y^\omega %
      =y^{\omega-1}(y^{\omega+1}x)^\omega y^{\omega+1}; \\
    (xy^\omega z)^{\omega+1} %
    &=x(y^\omega zx)^\omega y^\omega z %
      =x(zx y^\omega)^\omega y^\omega z %
      =x(zx y^\omega)^\omega z \\
    &=x(y^\omega zx)^\omega z %
      =(xy^\omega z)^\omega xz \\
    \therefore %
    (xy^\omega z)^{\omega+k} %
    &=(xy^\omega z)^\omega(xz)^k
      \text{ for every $k\ge1$} \\
    \therefore %
    (xy^\omega z)^{\omega+1} %
    &=(xy^\omega z)^\omega(xz)^{\omega+1} %
      \text{ and }
      (xy^\omega z)^\omega %
      =(xy^\omega z)^\omega(xz)^\omega \\
    \therefore %
    (xy^\omega z)^{\omega+k} %
    &=(xz)^k(xy^\omega z)^\omega %
      \text{ by symmetry}\\
    \therefore %
    (xy^\omega z)^{\omega+1}(xt^\omega z)^{\omega-1} %
    &=(xy^\omega z)^\omega(xz)\cdot(xz)^{\omega-1}(xt^\omega z)^\omega \\
    &=(xy^\omega z)^\omega(xz)^\omega(xt^\omega z)^\omega %
      =(xy^\omega z)^\omega(xt^\omega z)^\omega,
  \end{align*}
  which completes the proof.
\end{proof}

\begin{Cor}
  \label{c:pseudoidentities-for-BGn}
  The pseudovariety $(\pv{BG})_n$ satisfies the pseudoidentities
  \begin{align*}
    (xy^nz)^{\omega+1} %
    &=(xz)^{\omega+1}(xy^nz)^\omega %
      =(xy^nz)^\omega(xz)^{\omega+1} \\
    (xy^nz)^\omega %
    &=(xz)^\omega(xy^nz)^\omega %
      =(xy^nz)^\omega(xz)^\omega    
      .\qed
  \end{align*}
\end{Cor}

\section{Comparing several pseudovarieties}
\label{sec:comparison}

We start with a simple observation regarding the pseudovariety defined
by the Mal'cev product
\begin{math}
  \pv J\malcev\op x^n=1\cl.
\end{math}

\begin{Lemma}
  \label{l:Jm1=xn-in-BGn}
  \begin{math}
    \pv J\malcev\op x^n=1\cl\subseteq(\pv{BG})_n.
  \end{math}
\end{Lemma}

\begin{proof}
  The pseudoidentities
  \begin{math}
    x^n=1=y(x^ny)^{\omega-1}
  \end{math}
  are valid in~$\op x^n=1\cl$. Hence the Mal'cev product
  \begin{math}
    \pv J\malcev\op x^n=1\cl
  \end{math}
  satisfies the pseudoidentities
  \begin{align*}
    x^{\omega+n}
    &=(x^n)^{\omega+1}=(x^n)^\omega=x^\omega \\
    (x^ny)^\omega %
    &=\bigl(x^n\cdot y(x^ny)^{\omega-1}\bigr)^\omega %
      =\bigl(y(x^ny)^{\omega-1}\cdot x^n\bigr)^\omega %
      =(yx^n)^\omega,
  \end{align*}
  which proves the lemma.
\end{proof}

Note that the preceding proof may be adapted to show that the
pseudovariety
\begin{math}
  \pv J\malcev\op x^n=1\cl
\end{math}
satisfies the pseudoidentity
\begin{math}
  (ux)^\omega=(xu)^\omega
\end{math}
whenever the pseudoidentity $u=1$ holds in~$\op x^n=1\cl$. Such a
pseudoidentity
\begin{math}
  (ux)^\omega=(xu)^\omega
\end{math}
may however fail in~$(\pv{BG})_n$. For example, in case $n=2$, we may
take $u=yztytz$ and the resulting pseudoidentity fails in the
syntactic monoid of the language $(abcdbdc)^*$ over the alphabet
$\{a,b,c,d\}$, which lies in
\begin{math}
  \pv{BG}\cap\op x^3=x^2\cl,
\end{math}
as may be easily checked with the aid of computer calculations, and,
therefore also in~$(\pv{BG})_2$. In particular,
\begin{math}
  \pv J\malcev\op x^2=1\cl
\end{math}
does not contain $(\pv{BG})_2$. Below, we prove the stronger statement
that
\begin{math}
  \op 1\le x^n\cl\nsubseteq\pv A\malcev\op x^n=1\cl
\end{math}
(Corollary~\ref{c:<1LExn>-notin-Am1EQxn}).

For a language $L$, denote by $F(L)$ the set of all factors of words
from~$L$. Note that $F(L)$ is regular if so is~$L$. Given natural
numbers $k$ and~$\ell$, for shortness we denote by $w^{k+\ell*}$ the
set of all powers of the word $w$ whose exponent is of the form
$k+\ell n$ for some non-negative integer $n$.

\begin{Prop}
  \label{p:case2}
  Consider the following language over the alphabet $A=\{a,b,c\}$:
  \begin{displaymath}
    L_2=\Bigl(A^*\setminus F\bigl((abcacb)^*\bigr)\Bigr)
    \cup(abcacb)^{1+2*}.
  \end{displaymath}
  Then the syntactic ordered monoid of~$L_2$ belongs to~$\op 1\le
  x^2\cl$ and fails the pseudoidentity
  \begin{math}
    (xyzxzy)^{\omega+1}=(xyzxzy)^\omega.
  \end{math}
\end{Prop}

To prove the first part of Proposition~\ref{p:case2}, we establish the
following two lemmas.

\begin{Lemma}
  \label{l:uu-in-context-(abcacb)*}
  Suppose that $x,u,y$ are words such that $xu^2y$ belongs to the
  language $(abcacb)^*$. Then $|u|$ is a multiple of~6 and
  $xy\in(abcacb)^*$.
\end{Lemma}

\begin{proof}
  We first note that the case where $u=1$ is obvious while it is
  impossible that $u$ has length $1$ since no square of a letter
  belongs to $F\big((abcacb)^*\bigr)$. Hence, we may consider the
  prefix of length~2 of~$u$, which we denote $v$. Then $v$ must be one
  of the words $ab,bc,ca,ac,cb,ba$. Let $p,q$ be words such that $pvq$
  belongs to~$(abcacb)^*$. Note that, whatever the value of~$v$, its
  first letter can only appear in one position within $abcacb$, so
  that $|p|$ is completely determined modulo~6 by the value of~$v$.
  For instance, if $v=ab$, then we must have $|p|\equiv0\pmod6$.
 
  Let $u=vw$. By the above, since $xvwvwy$ is a power of~$abcacb$, we
  conclude that
  \begin{math}
    |x|\equiv|xvw|\pmod 6,
  \end{math}
  which yields
  \begin{math}
    |u|=|vw|\equiv0\pmod6.
  \end{math}
  Thus, whichever position in~$abcacb$ the factor $u^2$ starts in the
  power $xu^2y$ of~$abcacb$, the factor $y$ starts exactly in the same
  position. Hence, $xy$ belongs to~$(abcacb)^*$.
\end{proof}

\begin{Lemma}
  \label{l:ML-in-1LExn-case2}
  The syntactic ordered monoid of $L_2$ satisfies the inequality $1\le
  x^2$.
\end{Lemma}

\begin{proof}
  We must show that, if $p,q$ are words such that $pq\in L_2$, then
  $pu^2q\in L_2$ for every word $u\in A^*$.

  Suppose first that $pq$ belongs to the language %
  \begin{math}
    A^*\setminus F\bigl((abcacb)^*\bigr).
  \end{math}
  We claim that $pu^2q$ belongs to the same language, whence to~$L_2$.
  For this purpose, we argue by contradiction, assuming $pu^2q$ is a
  factor of some power of~$abcacb$, that is, there exist words $x,y$
  such that $xpu^2qy$ belongs to~$(abcacb)^*$. By
  Lemma~\ref{l:uu-in-context-(abcacb)*}, it follows that $xpqy$ also
  belongs to~$(abcacb)^*$, which contradicts the assumption that
  $pq$~does not belong to~$F\bigl((abcacb)^*\bigr)$.

  Hence, we may assume that $pq$ belongs to~$(abcacb)^{1+2*}$, so that
  there is an integer $k$ such that $pq=(abcacb)^{1+2k}$. If $pu^2q$
  is not in~$F\bigl((abcacb)^*\bigr)$, then it belongs to~$L_2$ and we
  are done. Thus, we assume that $pu^2q$ belongs
  to~$F\bigl((abcacb)^*\bigr)$ and we choose words $x,y$ such that
  \begin{math}
    xpu^2qy=(abcacb)^\ell
  \end{math}
  for some integer~$\ell$. By Lemma~\ref{l:uu-in-context-(abcacb)*},
  there is some integer $\ell'$ such that $|u|=6\ell'$ and
  \begin{math}
    xpqy=(abcacb)^{\ell-2\ell'}.
  \end{math}
  Since
  \begin{math}
    pq=(abcacb)^{1+2k},
  \end{math}
  and there are no nontrivial overlaps between the word $abcacb$ with
  itself, there exist integers $r,s$ such that $x=(abcacb)^r$ and
  $y=(abcacb)^s$. This yields the equality $\ell-2\ell'=r+s+1+2k$,
  that is,
  \begin{math}
    \ell-(r+s)=1+2k+2\ell',
  \end{math}
  which shows that
  \begin{math}
    pu^2q=(abcacb)^{1+2k+2\ell'}
  \end{math}
  is a word in $(abcacb)^{1+2*}$, whence also in~$L_2$.
\end{proof}

\begin{proof}[of Proposition~\ref{p:case2}]
  In view of Lemma~\ref{l:ML-in-1LExn-case2}, to complete the proof of
  Proposition~\ref{p:case2}, it remains to show that the syntactic
  monoid of~$L_2$ fails the pseudoidentity
  \begin{math}
    (xyzxzy)^{\omega+1}=(xyzxzy)^\omega.
  \end{math}
  Indeed, substituting the syntactic classes of $a,b,c$ respectively
  for the variables $x,y,z$, we obtain for $(xyzxzy)^{\omega+1}$ the
  syntactic class of a word of the form $(abcacb)^k$, where $k$~is
  odd, which belongs to~$L_2$, whereas for $\ell$~even,
  $(abcacb)^\ell$~does not belong to~$L_2$; hence, the value we obtain
  for $(xyzxzy)^\omega$ cannot be the same as
  for~$(xyzxzy)^{\omega+1}$.
\end{proof}

For $n\ge 3$, the argument is similar, but we need to work with a more
complicated word.

\begin{Prop}
  \label{p:general-case}
  Let $n\ge3$, $A=\{a,b\}$, $w=(b^{n-1}a)^{n-1}(ab)^{n-1}a^2$, and
  consider the following language:
  \begin{displaymath}
    L_n=\bigl(
    A^*\setminus F(w^*) \bigr) \cup w^{1+n*}.
  \end{displaymath}
  Then the syntactic ordered monoid of~$L_n$ belongs to~$\op 1\le
  x^n\cl$ and fails the pseudoidentity
  \begin{displaymath}
    \bigl((y^{n-1}x)^{n-1}(xy)^{n-1}x^2\bigr)^{\omega+1} %
    =\bigl((y^{n-1}x)^{n-1}(xy)^{n-1}x^2\bigr)^\omega.
  \end{displaymath}
\end{Prop}

The proof of Proposition~\ref{p:general-case} proceeds along the same
lines of the above proof for Proposition~\ref{p:case2}. The only point
where there is an essential difference is in the analogue of
Lemma~\ref{l:uu-in-context-(abcacb)*}, and that is the only detail
which we present here. The role of the number 6 is now played by
\begin{math}
  n^2+n=|(b^{n-1}a)^{n-1}(ab)^{n-1}a^2|.
\end{math}

\begin{Lemma}
  \label{l:un-in-context-...}
  Let $n\ge3$,
  \begin{math}
    w=(b^{n-1}a)^{n-1}(ab)^{n-1}a^2,
  \end{math}
  and suppose that $x,u,y$ are words such that $xu^ny$ belongs
  to~$w^*$. Then $n^2+n$ divides~$|u|$ and $xy$ also belongs to~$w^*$.
\end{Lemma}

\begin{proof}
  In case $u=1$, the result is immediate. Suppose $a^2b^2$ is not a
  factor of~$u^n$. Then, $u^n$ must be a factor of
  \begin{math}
    a(b^{n-1}a)^{n-1}(ab)^{n-1}a^2b,
  \end{math}
  which is easily seen to be impossible for a nonempty word~$u$.
  Hence, $a^2b^2$ must be a factor of~$u^n$ and, therefore, also
  of~$u^2$.

  Now, there is only one position where the factor $a^2b^2$ appears
  in~$w^2$, namely as %
  \begin{math}
    (b^{n-1}a)^{n-1}(ab)^{n-1}\cdot a^2b^2\cdot
    b^{n-2}a(b^{n-1}a)^{n-2}(ab)^{n-1}.
  \end{math}
  In $w^k$, two such consecutive positions are at distance $n^2+n$.
  Hence, whenever $pa^2b^2q$ is a power of~$w$, the value of $|p|$
  modulo $n^2+n$ is constant.
  
  Let $u^2=pa^2b^2q$, where $a^2b^2$ is not a factor of~$p$. By
  assumption, we know that $xu^ny\in w^*$. Since %
  \begin{math}
    xp\cdot a^2b^2\cdot qu^{n-2}y %
    =xup\cdot a^2b^2\cdot qu^{n-3}y
  \end{math}
  is a power of~$w$, we conclude from the preceding paragraph that
  \begin{math}
    |xp|\equiv|xup|\pmod {n^2+n},
  \end{math}
  which yields
  \begin{math}
    |u|\equiv0\pmod {n^2+n}.
  \end{math}
  Hence, in the factorization of~$xu^ny$ as a power of~$w$, the
  position in~$w$ where the factor $x$ ends must be followed, in a
  later occurrence of~$w$, precisely by the position where the factor
  $y$~starts. Thus, the factor $u^n$ may be removed to show that $xy$
  is also a power of~$w$.
\end{proof}

Combining Propositions~\ref{p:case2} and~\ref{p:general-case}, we
obtain the following result.

\begin{Cor}
  \label{c:<1LExn>-notin-Am1EQxn}
  For $n\ge2$, the pseudovariety of ordered monoids $\op 1\le x^n\cl$
  is not contained in
  \begin{math}
    \pv A\malcev\op x^n=1\cl.
  \end{math}
  In particular, $\op 1\le x^n\cl$~is contained in neither
  \begin{math}
    \pv J\malcev\op x^n=1\cl
  \end{math}
  nor
  \begin{math}
    \pv J*\op x^n=1\cl.
  \end{math}
\end{Cor}

\begin{proof}
  We observe that the pseudoidentities
  \begin{math}
    (xyzxzy)^{\omega+1}=(xyzxzy)^\omega
  \end{math}
  and
  \begin{displaymath}
    \bigl((y^{n-1}x)^{n-1}(xy)^{n-1}x^2\bigr)^{\omega+1} %
    =\bigl((y^{n-1}x)^{n-1}(xy)^{n-1}x^2\bigr)^\omega
  \end{displaymath}
  hold in the pseudovariety
  \begin{math}
    \pv A\malcev\op x^n=1\cl
  \end{math}
  respectively in case $n=2$ and $n>2$. Hence, it suffices to apply,
  respectively, Propositions~\ref{p:case2} and~\ref{p:general-case}.
\end{proof}

Another pseudovariety of interest is $(\pv{EJ})_n$, which is defined
to be the class of all finite monoids $M$ such that the submonoid
generated by $\{s^n: s\in M\}$ is \Cl J-trivial. This is clearly
contained in~$\pv{EJ}$, where only the submonoid generated by the
idempotents is required to be \Cl J-trivial, and satisfies the
pseudoidentity
\begin{math}
  x^{\omega+n}=x^\omega
\end{math}
so that
\begin{math}
  (\pv{EJ})_n\subseteq\pv B\op x^n=1\cl.
\end{math}

The diagram in Figure~\ref{fig:comparison} summarizes the known
inclusions between various pseudovarieties that we have considered so
far.
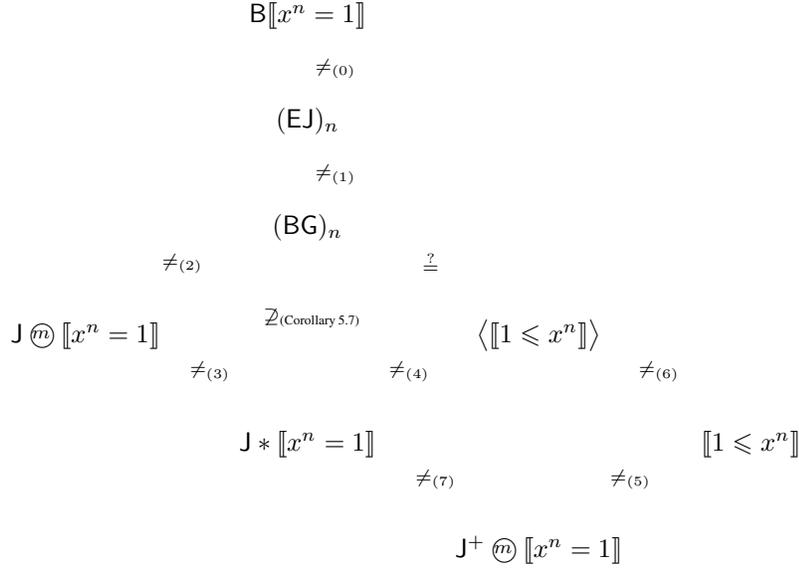
\begin{figure}[htp]
  \begin{displaymath}
    \xymatrix{
      & \pv B\op x^n=1\cl %
      \ar@{-}[d]^{\ne_{(0)}} %
      && \\
      & (\pv{EJ})_n %
      \ar@{-}[d]^{\ne_{(1)}} %
      && \\
      & (\pv{BG})_n %
      \ar@{-}[ld]_{\ne_{(2)}} %
      \ar@{-}[rd]^{\stackrel?=} %
      && \\
      \pv J\malcev\op x^n=1\cl %
      \ar@{-}[rd]^{\ne_{(3)}} %
      \ar@{..}[rr]^{\not\supseteq_{\text{(Corollary~\ref{c:<1LExn>-notin-Am1EQxn})}}} %
      && \bigl\langle\op 1\le x^n\cl\bigr\rangle %
      \ar@{-}[ld]_{\ne_{(4)}} %
      \ar@{-}[rd]^{\ne_{(6)}} %
      & \\
      & \pv J*\op x^n=1\cl %
      \ar@{-}[rd]^{\ne_{(7)}} %
      && \op 1\le x^n\cl %
      \ar@{-}[ld]_{\ne_{(5)}} %
      \\
      && \pv J^+\malcev\op x^n=1\cl }
  \end{displaymath}
  \caption{Comparison of several pseudovarieties}
  \label{fig:comparison}
\end{figure}
We next justify the strict inclusions depicted in the diagram in case $n\ge2$.
\begin{itemize}
\item[(0)]\label{item:ineq0} We define a monoid with zero by the
  following presentation:
  \begin{align*}
    M=\langle
    a,b:\ %
    &a^nb^na^n=a^n, %
      b^na^nb^n=b^n, \\
    &a^{n+1}=b^{n+1}=ab^ia=ba^ib=0\ (1\le i<n)
  \rangle.
  \end{align*}
  A simple calculation shows that $M$ has three regular \Cl J-classes,
  two containing only the idempotents $1$ and $0$, respectively, and
  the third one containing the idempotents $a^ib^na^{n-i}$ and
  $b^ia^nb^{n-i}$ ($1\le i\le n$). The product of any two distinct
  idempotents different from~$1$ is~$0$, so that $M$ belongs
  to~\pv{BG}, and $M$~is aperiodic, whence
  \begin{math}
    M\in\pv B\op x^n=1\cl.
  \end{math}
  On the other hand,
  \begin{math}
    (a^nb^n)^\omega=a^nb^n\ne b^na^n=(b^na^n)^\omega,
  \end{math}
  which shows that $M$ does not belong
  to~$(\pv{EJ})_n$.
\item[(1)] Consider the monoid with zero given by the presentation
  \begin{displaymath}
    \qquad\quad
    M=\langle a,b: %
    a^nba^n=a^n, %
    ba^nb=b, %
    ba^ib=0\ (0\le i < n), %
    a^{n+1}=0 \rangle.
  \end{displaymath}
  It is easy to see that $M$ consists of the elements $1$, $a^i$
  ($1\le i<n$), $0$, which form singleton \Cl J-classes, together
  with $a^iba^j$ ($0\le i,j\le n$), which constitutes a \Cl J-class
  whose idempotents are the elements $a^iba^j$ for which $i+j=n$. The
  $n$th powers are the idempotents and $a^n$, and form a
  submonoid of~$M$ which is \Cl J-trivial, that is, $M$ belongs
  to~$(\pv{EJ})_n$. On the other hand,
  \begin{math}
    M\notin(\pv{BG})_n
  \end{math}
  since the idempotents $a^nb$ and $ba^n$ are distinct.
\item[(2),(4)] These follow from
  Corollary~\ref{c:<1LExn>-notin-Am1EQxn}.
\item[(3)] See Remark~\ref{r:J*H-vs-JmH}.
\item[(5)] This follows from Lemma~\ref{l:separate-J+mHn-from-1LEQxn}.
  Alternatively, it also follows from (4) since, for the two other
  sides of the diamond involving the inclusions (4) and (5), one goes
  up by taking the pseudovariety of monoids generated by a
  pseudovariety of ordered monoids.
\item[(6)] The equality %
  \begin{math}
    \bigl\langle\op 1\le x^n\cl\bigr\rangle=\op 1\le x^n\cl
  \end{math}
  means that, for every
  \begin{math}
    (M,{\le})\in\op 1\le x^n\cl,
  \end{math}
  we have
  \begin{math}
    (M,{=})\in\op 1\le x^n\cl,
  \end{math}
  so that
  \begin{math}
    \op 1\le x^n\cl=\op x^n=1\cl,
  \end{math}
  which contradicts~(5).
\item[(7)] The argument is similar to that given for~(6) and is
  omitted.
\end{itemize}

For (4), we may also prove the following stronger result.

\begin{Prop}
  \label{p:<1LExn>-not-J*V}
  Whenever $n\ge2$, there is no pseudovariety \pv V such that %
  \begin{math}
    \bigl\langle\op 1\le x^n\cl\bigr\rangle %
    =\pv J*\pv V.
  \end{math}
\end{Prop}

\begin{proof}
  Assume to the contrary that there is such a pseudovariety \pv V. In
  particular, we have
  \begin{math}
    \pv J*\pv V\subseteq\pv{BG}.
  \end{math}
  If
  \begin{math}
    \pv{Sl}\subseteq\pv V,
  \end{math}
  then it follows that
  \begin{math}
    \pv{Sl}*\pv{Sl}\subseteq\pv{BG}.
  \end{math}
  However, it is easy to see that the monoid consisting of the
  two-element left-zero semigroup with an identity adjoined satisfies
  the identities defining $\pv{Sl}*\pv{Sl}$
  \cite[Exercise~10.3.7]{Almeida:1994a}, while it is not in~\pv{BG}.
  Hence, \pv{Sl} is not contained in the pseudovariety of monoids \pv
  V, which implies that \pv V is a pseudovariety of groups. On the
  other hand, we know that %
  \begin{math}
    (\pv J*\pv V)\cap\pv G=\pv V\cap\pv G
  \end{math}
  (see \cite[Proposition~10.1.7]{Almeida:1994a}). We conclude that
  \begin{displaymath}
    \pv V %
    =\pv V\cap\pv G %
    =(\pv J*\pv V)\cap\pv G %
    =\bigl\langle\op 1\le x^n\cl\bigr\rangle\cap\pv G %
    =\op x^n=1\cl,
  \end{displaymath}
  and so
  \begin{math}
    \op 1\le x^n\cl\subseteq\pv J*\pv V=\pv J*\op x^n=1\cl,
  \end{math}
  which contradicts Corollary~\ref{c:<1LExn>-notin-Am1EQxn}.
\end{proof}

We can also prove the following result for the Mal'cev product.

\begin{Prop}
  \label{p:when-<1LExn>=JmV}
  If there is some pseudovariety of monoids \pv V such that
  \begin{math}
    \bigl\langle\op 1\le x^n\cl\bigr\rangle=\pv J\malcev\pv V,
  \end{math}
  then %
  \begin{math}
    \pv J\malcev\op x^n=1\cl %
    \subseteq\bigl\langle\op 1\le x^n\cl\bigr\rangle.
  \end{math}
\end{Prop}

\begin{proof}
  Suppose that the equality %
  \begin{math}
    \bigl\langle\op 1\le x^n\cl\bigr\rangle=\pv J\malcev\pv V %
  \end{math}
  holds. It follows that
  \begin{equation}
    \label{eq:when-<1LExn>=JmV}
    \op x^n=1\cl %
    =\pv G\cap\bigl\langle\op 1\le x^n\cl\bigr\rangle %
    =(\pv J\malcev\pv V)\cap\pv G %
    =\pv V\cap\pv G,
  \end{equation}
  where only the last equality remains to be justified. The inclusion
  $\supseteq$ is a consequence of
  \begin{math}
    \pv V\subseteq\pv J\malcev\pv V.
  \end{math}
  For the reverse inclusion, suppose that $G$~is a group from~$\pv
  J\malcev\pv V$. Then, there is a relational morphism $\mu:G\to V$
  onto some $V\in\pv V$ such that, for every idempotent $e$ from~$V$,
  $\mu^{-1}(e)$ is a semigroup from~\pv J. Since $G$~is a group and
  $\pv J\cap\pv G$ is the trivial pseudovariety, consisting only of
  singleton monoids, we deduce that $\mu^{-1}(e)=\{1\}$ for every
  idempotent $e$ from~$V$. Consider now the canonical factorization
  of~$\mu$: $\mu$ may be viewed as a submonoid of $G\times V$; we
  denote $\varphi$ and $\psi$ respectively the restrictions to $\mu$
  of the first and second component projections of the product
  $G\times V$; then, as a relation, $\mu=\varphi^{-1}\psi$. Since
  $\varphi$ is onto, there is a subgroup $H$ of~$\mu$ such that
  $\varphi(H)=G$ (see, for instance,
  \cite[Proposition~4.1.44]{Rhodes&Steinberg:2009qt}). Consider the
  subgroup $K=\psi(H)$ of~$V$. Note that
  \begin{displaymath}
    \psi^{-1}(1_K) %
    =\{(g,1_K)\in\mu:g\in G\} %
    =\mu^{-1}(1_K)\times\{1_K\} %
    =\{(1,1_K)\}.
  \end{displaymath}
  Thus, the group kernel of the homomorphism $\psi$ is trivial, so
  that $H$ and $K$ are isomorphic, whence $G$ belongs to~\pv V since
  so does~$K$.

  From (\ref{eq:when-<1LExn>=JmV}), we know that $\op x^n=1\cl$ is
  contained in~\pv V, which finally entails the required inclusion %
  \begin{math}
    \pv J\malcev\op x^n=1\cl %
    \subseteq\bigl\langle\op 1\le x^n\cl\bigr\rangle.
  \end{math}
\end{proof}

Although we are interested mainly in the comparison of several
pseudovarieties and the computation of
\begin{math}
  \bigl\langle\op 1\le x^n\cl\bigr\rangle
\end{math}
in case $n$~is an integer with $n\ge 2$, the cases
$n=1$ and $n=\omega$ are also of great interest. In fact, they have
deserved considerable attention in the literature.

The case $n=1$ is that of the pseudovariety
\begin{math}
  \pv J^+=\op 1\le x\cl.
\end{math}
It has been shown to be equivalent to a celebrated theorem of
\cite{Simon:1975} that
\begin{math}
  \langle \pv J^+\rangle=\pv J.
\end{math}
A direct algebraic proof of this fact has been given
by~\cite{Straubing&Therien:1988a}. In this case, we also have
\begin{math}
  \pv B\op x=1\cl\ne(\pv{EJ})_1=(\pv{BG})_1=\pv J=\pv J*\op x=1\cl.
\end{math}

In case $n=\omega$, we know that
\begin{math}
  \pv J*\pv G=\pv{BG}
\end{math}
(see \cite{Pin:1995}) and
\begin{math}
  \pv J^+\malcev\pv G=\op 1\le x^\omega\cl
\end{math}
\cite[Theorem~2.7]{Pin&Weil:1994c}. There are proofs of these facts in
the literature that depend on a deep result of~\cite{Ash:1991}. In the
case of the first equality, an alternative ``constructive'' proof can
be found in~\cite{Auinger&Steinberg:2005c}. Hence, by
Corollary~\ref{c:BPolH}, the equality
\begin{math}
  \langle\op 1\le x^\omega\cl\rangle=\pv{BG}
\end{math}
holds.

\section{Algebraically provable inequalities}
\label{sec:provable-ineqs}

An inequality $u'\le v'$ is said to be a \emph{direct consequence} of
the inequality $u\le v$ if $u,v\in\Om AM$, $u',v'\in\Om BM$, and there
is a continuous homomorphism
\begin{math}
  \varphi:\Om AM\to\Om BM
\end{math}
such that
$u'=\varphi(u)$ and $v'=\varphi(v)$.

By an \emph{algebraic proof} of an inequality $u\le v$ from a set
$\Sigma$ of inequalities we mean a pair of finite sequences of
pseudowords $(x_iy_iz_i)_{i=1,\ldots,m}$ and $(t_i)_{i=1,\ldots,m}$
such that $u=x_1y_1z_1$, $v=x_mt_mz_m$, each inequality $y_i\le t_i$
is a direct consequence of some inequality from~$\Sigma$, and
\begin{math}
  x_it_iz_i=x_{i+1}y_{i+1}z_{i+1}
\end{math}
($i=1,\ldots,m-1$). In case there exists such a sequence, we also say
that the inequality $u\le v$ is \emph{algebraically provable}
from~$\Sigma$. A pseudoidentity $u=v$ is \emph{algebraically provable}
from~$\Sigma$ if both inequalities $u\le v$ and $v\le u$ have that
property.

Note that, in particular, $u\le v$ is algebraically provable from
$1\le x^n$ if and only if $v$ may be obtained from $u$ by a finite
sequence of insertions of factors of the form $w^n$.

Since our aim is to show that $\langle\op 1\le
x^n\cl\rangle=(\pv{BG})_n$ and we already know that the inclusion from
left to right holds, by Reiterman's theorem,
see~\cite{Reiterman:1982}, this amounts to showing that every
pseudoidentity valid in the pseudovariety~$\langle\op 1\le
x^n\cl\rangle$ is also valid in~$(\pv{BG})_n$. The following
proposition gives a key connection between inequalities provable
from~$1\le x^n$ and pseudoidentities valid in~$(\pv{BG})_n$.

\begin{Prop}
  \label{p:pseudoidentities-for-BGn}
  The pseudovariety $(\pv{BG})_n$ satisfies the following
  pseudoidentities:
  \begin{enumerate}[(a)]
  \item\label{item:pseudoidentities-for-BGn-1} %
    \begin{math}
      v^{\omega+1}=u^{\omega+1} v^\omega=v^\omega u^{\omega+1}
    \end{math}
    whenever the inequality $u\le v$ is algebraically provable from
    $1\le x^n$;
  \item\label{item:pseudoidentities-for-BGn-2} %
    \begin{math}
      v^\omega=u^\omega v^\omega=v^\omega u^\omega
    \end{math}
    whenever the inequality $u\le v$ is algebraically provable from
    $1\le x^n$;
  \item\label{item:pseudoidentities-for-BGn-3}
    \begin{math}
      u^{\omega+1}=v^{\omega+1}
    \end{math}
    whenever the pseudoidentity $u=v$ is algebraically provable from
    $1\le x^n$.
  \end{enumerate}
\end{Prop}

\begin{proof}
  (\ref{item:pseudoidentities-for-BGn-1}) Consider an algebraic proof
  of $u\le v$ from~$1\le x^n$, given by a pair of sequences of
  pseudowords $(x_iz_i)_{i=1,\ldots,m}$ and $(y_i)_{i=1,\ldots,m}$.
  Note that, for each $i$, since $1\le y_i$ is a direct consequence of
  $1\le x^n$, there is $t_i$ such that $y_i=t_i^n$. Let $u_i=x_iz_i$
  for $i=1,\ldots,m$ and $u_{m+1}=x_mt_m^nz_m=v$. We prove, by
  induction on~$i$, that
  \begin{equation}
    \label{eq:pseudoidentities-for-BGn-target-1}
    (\pv{BG})_n \text{ satisfies } u_i^{\omega+1}=u^{\omega+1}u_i^\omega.
  \end{equation}
  For $i=m+1$, this gives one of the pseudoidentities
  in~(\ref{item:pseudoidentities-for-BGn-1}), the other one being
  obtained dually.

  Since $u_1=u$, (\ref{eq:pseudoidentities-for-BGn-target-1}) is
  immediate for $i=1$. Suppose that
  (\ref{eq:pseudoidentities-for-BGn-target-1})~holds for a certain
  $i\le m$. Then, in view of Corollary~\ref{c:pseudoidentities-for-BGn}
  and the induction hypothesis, $(\pv{BG})_n$ satisfies the following
  pseudoidentities:
  \begin{align*}
    u_{i+1}^{\omega+1} %
    &=(x_it_i^nz_i)^{\omega+1} %
      =(x_iz_i)^{\omega+1}(x_it_i^nz_i)^\omega %
      =u_i^{\omega+1}u_{i+1}^\omega \\
    &=u^{\omega+1} u_i^\omega u_{i+1}^\omega %
      =u^{\omega+1}(x_iz_i)^\omega(x_it_i^nz_i)^\omega \\
    &=u^{\omega+1}(x_it_i^nz_i)^\omega
      =u^{\omega+1} u_{i+1}^\omega,
  \end{align*}
  which completes the induction step for the proof
  of~(\ref{eq:pseudoidentities-for-BGn-target-1}).

  (\ref{item:pseudoidentities-for-BGn-2}) This can be established by a
  slight modification of the proof
  of~(\ref{item:pseudoidentities-for-BGn-1}), namely by replacing
  (\ref{eq:pseudoidentities-for-BGn-target-1}) by $(\pv{BG})_n$
  satisfies $u_i^\omega=u^\omega u_i^\omega$.
  
  (\ref{item:pseudoidentities-for-BGn-3}) From
  (\ref{item:pseudoidentities-for-BGn-1})
  and~(\ref{item:pseudoidentities-for-BGn-2}), it follows that
  $(\pv{BG})_n$ satisfies the pseudoidentities
  \begin{displaymath}
    v^{\omega+1} %
    =u^{\omega+1}v^\omega %
    =u^{\omega+1}u^\omega v^\omega %
    =u^{\omega+1}.\popQED
  \end{displaymath}
\end{proof}

\section{More general proofs}
\label{sec:proofs}

Let $\Sigma$ be a set of inequalities $u\le v$ with $u$ and~$v$
pseudowords over some finite alphabet. We are interested in allowing
more general proofs of the validity of inequalities in the
pseudovariety $\op\Sigma\cl$ than those considered in
Section~\ref{sec:provable-ineqs}. For simplicity, we fix the finite
set $A$ of variables on which we consider such provable inequalities.
The definitions below extend to the case of inequalities those
previously considered by the authors for pseudoidentities,
see~\cite{Almeida&Klima:2017a}.

For each ordinal $\alpha$, we define recursively a set $\Sigma_\alpha$ of
inequalities over~$A$ as follows:
\begin{itemize}
\item $\Sigma_0$ consists of all diagonal pairs $(w,w)$, with $w\in\Om
  AM$, together with all pairs of the form $(x\varphi(u)y,
  x\varphi(v)y)$ such that $u\le v$ is an inequality from~$\Sigma$,
  say with $u,v\in\Om BM$,
  \begin{math}
    \varphi:\Om BM\to\Om AM
  \end{math}
  is a continuous homomorphism, and $x,y\in\Om AM$;
\item $\Sigma_{2\alpha+1}$ is the transitive closure of the binary
  relation $\Sigma_{2\alpha}$;
\item $\Sigma_{2\alpha+2}$ is the topological closure of the relation
  $\Sigma_{2\alpha+1}$ in the space $\Om AM\times\Om AM$;
\item if $\alpha$~is a limit ordinal, then
  \begin{math}
    \Sigma_\alpha=\bigcup_{\beta<\alpha}\Sigma_\beta.
  \end{math}
\end{itemize}
Note that $\Sigma_1$ consists of the algebraically provable
inequalities and that, if
\begin{math}
  \Sigma_{\alpha+2}=\Sigma_{\alpha},
\end{math}
then $\Sigma_\alpha$ is both transitive and topologically closed, so
that
\begin{math}
  \Sigma_\beta=\Sigma_\alpha
\end{math}
for every ordinal $\beta$ with
$\beta\ge\alpha$. Since \Om AM is a metric space, such a condition
must hold for $\alpha$ at most the least uncountable ordinal (see
\cite[Proposition~3.1]{Almeida&Klima:2017a} for a justification in the
unordered case, which applies equally well to the ordered case).
Hence, the union
\begin{math}
  \tilde{\Sigma}=\bigcup_{\alpha}\Sigma_\alpha
\end{math}
defines a transitive closed binary relation on~\Om AM.

Consider a binary relation $\theta$ on~\Om AM. We say that $\theta$~is
\emph{stable} if
\begin{math}
  (u,v)\in\theta
\end{math}
and $x,y\in\Om AM$ implies
\begin{math}
  (xuy,xvy)\in\theta.
\end{math}
We also say that $\theta$~is \emph{fully
  invariant} if, for every continuous endomorphism $\varphi$ of~$\Om
AM$ and $(u,v)\in\theta$, we have
\begin{math}
  (\varphi(u),\varphi(v))\in\theta.
\end{math}

The next result is the order analog of
\cite[Proposition~3.1]{Almeida&Klima:2017a}, with a proof following
the very same lines, which is therefore omitted.

\begin{Prop}
  \label{p:gen-provable-inequalities}
  The relation $\tilde{\Sigma}$ is a fully invariant closed stable
  quasiorder on~$\Om AM$. For every
  \begin{math}
    (u,v)\in\tilde{\Sigma},
  \end{math}
  the inequality $u\le v$~is valid in~$\op \Sigma\cl$.
\end{Prop}

The pairs from~$\tilde{\Sigma}$, which are viewed as inequalities, are
said to be \emph{provable from~$\Sigma$}. We also say that a
pseudoidentity $u=v$ is \emph{provable from~$\Sigma$} if so are both
inequalities $u\le v$ and~$v\le u$.

An alternative way of looking at proofs, which is equivalent in the
sense of capturing the same provable inequalities, is to consider a
transfinite sequence of inequalities in which in each step we allow
one of the inequalities of~$\Sigma_0$, we take $u\le w$ if there are
two previous steps of the form $u\le v$ and $v\le w$, or we take $u\le
v$ provided there is a sequence of earlier steps $(u_n\le v_n)_n$ with
$u=\lim u_n$ and $v=\lim v_n$. The last step in such a proof should be
the inequality to be proved.

Several examples of such proofs, can be found in the proofs of
Lemma~\ref{l:BGn} and Proposition~\ref{p:upperbound-1LEQxn}. The
following is the order analog
of~\cite[Conjecture~3.2]{Almeida&Klima:2017a}. Note that ample evidence
for the unordered case is presented in~\cite{Almeida&Klima:2017a}.

\begin{Conjecture}
  \label{cj:provable}
  An inequality $u\le v$ of elements from~\Om AM is provable
  from~$\Sigma$ if and only if $\op \Sigma\cl$ satisfies $u\le v$.
\end{Conjecture}

Taking into account the analogues of the results
of~\cite[Section~3.8]{Almeida:1994a} for inequalities, the conjecture
is equivalent to showing that $\tilde{\Sigma}$ is a profinite relation
in the sense of~\cite[Section~3.1]{Rhodes&Steinberg:2009qt}, that is
that $\tilde{\Sigma}$ is a closed stable quasiorder such that the
quotient by the congruence obtained by taking the intersection with
the dual of~$\tilde{\Sigma}$ is a profinite monoid. There seems to be
no obvious way of establishing such a property.

\section{Pseudoidentities provable from \texorpdfstring{$1\le
    x^n$}{1<=xn}}
\label{sec:provable-vs-BGn}

The aim of this section is to show that, at least under suitable
hypotheses, pseudoidentities provable from $1\le x^n$ are valid
in~$(\pv{BG})_n$. We start by extending
Proposition~\ref{p:pseudoidentities-for-BGn}.

\begin{Prop}
  \label{p:more-pseudoidentities-for-BGn}
  The statement of Proposition~\ref{p:pseudoidentities-for-BGn}
  remains true if the adverb ``algebraically'' is removed.
\end{Prop}

\begin{proof}
  We only handle the analogue of
  part~(\ref{item:pseudoidentities-for-BGn-1})
  as~(\ref{item:pseudoidentities-for-BGn-2}) is similar and the proof
  of (\ref{item:pseudoidentities-for-BGn-3})~does not require any
  changes. So, let $\Sigma$ consist of the single inequality $1\le
  x^n$ and consider a finite alphabet and the binary relations
  $\Sigma_\alpha$ over~\Om AM. We prove by transfinite induction
  on~$\alpha$ that, whenever
  \begin{math}
    (u,v)\in\Sigma_\alpha,
  \end{math}
  $(\pv{BG})_n$ satisfies the pseudoidentity
  \begin{math}
    u^{\omega+1}v^\omega=v^{\omega+1}.
  \end{math}
  The cases of $\alpha=0$ and $\alpha=2\beta+1$, that is, respectively
  inequalities of the form
  \begin{math}
    xz\le xy^nz
  \end{math}
  or that obtained from inequalities from~$\Sigma_{2\beta}$ by
  transitivity, are already essentially handled in the proof of
  Proposition~\ref{p:pseudoidentities-for-BGn}. For the case where
  $\alpha=2\beta+2$, we consider a sequence $(u_k,v_k)_k$
  from~$\Sigma_{2\beta+1}$ converging to the limit $(u,v)$. By the
  induction hypothesis, $(\pv{BG})_n$ satisfies each of the
  pseudoidentities
  \begin{math}
    u_k^{\omega+1}v_k^\omega=v_k^{\omega+1}.
  \end{math}
  Hence, taking limits on both sides, we conclude that $(\pv{BG})_n$
  also satisfies
  \begin{math}
    u^{\omega+1}v^\omega=v^{\omega+1}.
  \end{math}
  Finally, in case $\alpha$~is a limit ordinal, the induction step is
  immediate since
  \begin{math}
    \Sigma_\alpha=\bigcup_{\beta<\alpha}\Sigma_\beta.
  \end{math}
\end{proof}

We say that a pseudoword $w\in\Om AM$ has a certain property
over~$(\pv{BG})_n$ if that property is verified by $\pi(w)$ where
\begin{math}
  \pi:\Om AM\to\Om A{}(\pv{BG})_n
\end{math}
is the unique continuous homomorphism sending each $a\in A$ to itself.

\begin{Thm}
  \label{t:provable-pseudoidentities}
  Let $u=v$ be a pseudoidentity provable from~$1\le x^n$.
  \begin{enumerate}[(a)]
  \item\label{item:t:provable-pseudoidentities-1} If $u$ and $v$ are
    group elements over~$(\pv{BG})_n$, then $(\pv{BG})_n$ satisfies
    the pseudoidentity $u=v$.
  \item\label{item:t:provable-pseudoidentities-2} If there are
    pseudowords $w$ and $z$ such that the pseudoidentity
    $u=(wz)^\omega w$ is valid in~$(\pv{BG})_n$, then so are the
    pseudoidentities
    \begin{math}
      u=v(zw)^\omega=(wz)^\omega v.
    \end{math}
  \item\label{item:t:provable-pseudoidentities-3} If $u$ and $v$ are
    regular over~$(\pv{BG})_n$, then $(\pv{BG})_n$ satisfies the
    pseudoidentity $u=v$.
  \end{enumerate}
\end{Thm}

\begin{proof}
  (\ref{item:t:provable-pseudoidentities-1}) By
  Proposition~\ref{p:more-pseudoidentities-for-BGn}, from the
  hypothesis that $u=v$~is provable from $1\le x^n$, we deduce that
  $(\pv{BG})_n$ satisfies
  \begin{math}
    u^{\omega+1}=v^{\omega+1}.
  \end{math}
  But, since we are assuming that $(\pv{BG})_n$ satisfies
  \begin{math}
    u^{\omega+1}=u
  \end{math}
  and
  \begin{math}
    v^{\omega+1}=v,
  \end{math}
  it follows that it also satisfies $u=v$.

  (\ref{item:t:provable-pseudoidentities-2}) From the assumption that
  $u=v$ is provable from~$1\le x^n$ it follows that so is $uz=vz$.
  Part~(\ref{item:t:provable-pseudoidentities-1}) yields that the
  pseudoidentities
  \begin{math}
    (vz)^{\omega+1}=(uz)^{\omega+1}=(wz)^{\omega+1}=uz
  \end{math}
  hold in~$(\pv{BG})_n$. Hence, so do
  \begin{math}
    (uz)^\omega=(vz)^\omega=(wz)^\omega
  \end{math}
  and the following pseudoidentities:
  \begin{align*}
    uz
    &=(uz)^{\omega+1}
      =(vz)^{\omega+1}
      =vz(uz)^\omega
      =vz(wz)^\omega
    \\
    \therefore %
    u
    &=uz\cdot(wz)^{\omega-1}w
      =vz(wz)^\omega\cdot(wz)^{\omega-1}w
      =v(zw)^\omega.
  \end{align*}
  The proof that $(\pv{BG})_n$ satisfies the pseudoidentity
  \begin{math}
    u=(wz)^\omega v
  \end{math}
  is dual.

  (\ref{item:t:provable-pseudoidentities-3}) Since $u$ is regular
  over~$(\pv{BG})_n$, there are pseudowords $w$ and $z$ such that the
  pseudoidentity
  \begin{math}
    u=(wz)^\omega w
  \end{math}
  holds in~$(\pv{BG})_n$. From
  part~(\ref{item:t:provable-pseudoidentities-2}), it follows that $u$
  is both \Cl R and \Cl L below $v$ over~$(\pv{BG})_n$. By symmetry,
  we conclude that $u$ and $v$ lie in the same \Cl H-class
  over~$(\pv{BG})_n$. Since $u$ is \Cl R-equivalent to~$(wz)^\omega$
  over~$(\pv{BG})_n$ and
  \begin{math}
    u=(wz)^\omega v
  \end{math}
  holds in~$(\pv{BG})_n$, so does $u=v$.
\end{proof}

Theorem~\ref{t:provable-pseudoidentities} may be viewed as a hint that
the equality of pseudovarieties
\begin{math}
  \langle\op 1 \le x^n \cl\rangle=(\pv{BG})_n
\end{math}
may hold. Should Conjecture~\ref{cj:provable} hold for the inequality
$1\le x^n$, the evidence for the equality is even more compelling. At
present, we must leave it as an open problem.

Another natural and weaker question is whether
\begin{math}
  \pv J\malcev\op x^n=1\cl
\end{math}
is contained in
\begin{math}
  \langle\op 1\le x^n\cl\rangle.
\end{math}
We have no further partial results in this direction than those that
follow from Theorem~\ref{t:provable-pseudoidentities}.

Other questions worth investigating concerning the pseudovarieties in
Figure~\ref{fig:comparison} involve the corresponding relatively free
profinite monoids. For instance, using the representation theorem for
semidirect products \cite[Theorem~10.2.3]{Almeida:1994a}, the fact
that \Om AJ is countable for every finite set $A$
\cite[Proposition~8.2.1]{Almeida:1994a}, and the local finiteness of
the Burnside pseudovariety~$\op x^n=1\cl$, see \cite{Zelmanov:1991},
we deduce that
\begin{math}
  \Om A{}(\pv J*\op x^n=1\cl)
\end{math}
is also countable in case $A$~is finite. We do not know if a similar
property holds for any of the pseudovarieties
\begin{math}
  \pv J\malcev\op x^n=1\cl$, $(\pv{BG})_n,
\end{math}
or perhaps even
\begin{math}
  \pv B\op x^n=1\cl.
\end{math}

\acknowledgements
\label{sec:ack}

The authors would like to thank the referees for their suggestions and
comments, which contributed to an improved presentation of our
results.

\bibliographystyle{abbrvnat}
\bibliography{sgpabb,ref-sgps}

\end{document}